\documentclass[titlepage,12pt]{article}

\usepackage{graphicx}%
\usepackage{amsmath,amssymb,amsfonts}%
\usepackage{amsthm}%
\usepackage{mathrsfs}%
\usepackage{xcolor}%
\usepackage{tikz}%
\usepackage{url}
\usepackage{fancybox}

\theoremstyle{change}
\newtheorem{proclaim}{PROCLAIM}[section]
\newtheorem{prop}[proclaim]{Proposition}
\newtheorem{rem}[proclaim]{Remark}
\newtheorem{dfn}[proclaim]{Definition}
\newtheorem{thm}[proclaim]{Theorem}

\usepackage{graphicx}
\usepackage{url}

\usepackage{enumitem}

\newcommand{\bR}{\mathbb{R}}
\newcommand{\R}{\mathbb{R}}
\newcommand{\bRn}{\mathbb{R}^n}
\newcommand{\bN}{\mathbb{N}}
\newcommand{\N}{\mathbb{N}}

\def\P{{\mathbf P}}
\def\X{{\mathbf X}}

\def\val{\mathop{\rm val}}
\def\ri{\mathop{\rm ri}}
\def\dom{\mathop{\rm dom}}
\def\epi{\mathop{\rm epi}}

\begin{document}

\baselineskip=15pt

\begin{titlepage}
	\vglue 0.5cm
	\begin{center}
		\begin{large}
			{\bf Subgradient evolution of value functions in discrete-time optimal control}
			
			\smallskip
		\end{large}
		\vglue 1.truecm
		\begin{tabular}{clclcl}
			\begin{large} {\sl Julio Deride
			} \end{large}&&
			\begin{large} {\sl Cristopher Hermosilla
			} \end{large}&&
			\begin{large} {\sl Mattia Solla
			} \end{large}&\\
			Department of Mathematics&&			Department of Mathematics&&			Department of Mathematics&\\
			Universidad T.F.Santa Mar\'ia&&			Universidad T.F.Santa Mar\'ia&&			Universidad T.F.Santa Mar\'ia&\\
			julio.deride@usm.cl&&cristopher.hermosill@usm.cl&&mattia.solla@usm.cl& 
		\end{tabular}
	\end{center}
	\vskip 0.5truecm
	\noindent 
		{\bf Abstract}.\quad In this paper we investigate how the subgradients of the value function of a discrete-time convex Bolza problem evolve over time. In particular, we develop a
		discrete-time version of the characteristic method introduced by Rockafellar and Wolenski in the 2000s, by showing that the time-evolution of the subgradients of the value functions can be associated with trajectories of a discrete-time Hamiltonian system. To do so, we first prove that the value function has a dual counterpart, which corresponds to the conjugate of the value function of a suitable dual problem. We finally make a discussion about the qualification conditions we require for our results, showing in particular that classical problems, such as the Liner-Quadratic regulator, satisfy these hypotheses.
	\vskip 1.5truecm
	\halign{&\vtop{\parindent=0pt
			\hangindent2.5em\strut#\strut}\cr
		\hglue 1.50cm   problems, Linear-Quadratic regulator, Mixed constraints.\hfill\break\cr
		{\bf Date}:\quad \ \today \cr}
	\end{titlepage}
	\baselineskip=15pt

\section{Introduction}
This paper is concerned with Bolza problems associated with a discrete-time dynamical system, that is, we aim at studying the optimization problem
\[\text{Minimize }\sum_{t=\tau+1}^TL_t(x_{t-1},\Delta x_t)+g(x_T)\]
over all $\{x_t\}_{t=\tau}^T\subset\bR^n$ such that $x_\tau=\xi$. Here $\tau$ and $\xi$ are the initial data of the problem, which for our purposes are considered as inputs of the problem. The value $T\in\N$ is the final horizon, which is assumed to be fixed all along this manuscript.

Our goal is to study the value function of the problem, that is the value  of the optimization problem, denoted $V(\tau,\xi)$, seen as a function of the initial data $(\tau,\xi)$, and in particular to understand the evolution over time of the subgradients of the function $\xi\mapsto V(\tau,\xi)$.

To undertake this task, we focus on Bolza problems of convex type (see Assumption \eqref{hyp:convex} below). This fact provides two key points for our analysis. The first one, being that the value function turns our to be convex with respect to the state variable, and the second one being that this structural property of the value function allows to study convex Bolza problems from a  duality viewpoint. Indeed, in the 1970s, in a series of papers \cite{Roc70,Roc71,Roc72,Roc73,Roc76}, Rockafellar set the foundations of a duality theory for this type of problems with underlying continuous-time systems. The main advantage when working in this context is that a convex optimal control problem can be paired with a dual problem, which turns out also to be  a convex Bolza problem. A sample result that can be obtained from the theory is that coextremals are minimizers of the dual problem; see for instance \cite[Section 10]{Roc70}. This in particular leads to necessary optimality conditions, which in this case are habitually sufficient too, to be given in terms of a Hamiltonian system and also in the form of a Maximum Principle; see for instance \cite{BecHer22,PenPer14,Roc70,Roc76,Roc89}. It is worth mentioning that considerable attention has been put on studying the value function in the light of this duality theory for problems without pathwise constraints; see for instance \cite{Goe04b,Goe04a,Goe05,Goe05a,HerWol17,HerWol19,Roc04,RocGoe08,RocWol00,RocWol00b}. Some  contributions for the state constrained case have been done by one of the authors (see \cite{HerWol17,HerWol19}), but there are still several questions to be addressed. In particular, to understand how these results transpose to deterministic and/or stochastic discrete-time systems.

It is worth mentioning that  Bolza problem with possibly non-finite Lagrangians allow us to recover classical control problems such as the linear quadratic optimal control problems (see \cite{Roc87,GoeSub07,HerWol17}) or Linear-convex problems with mixed constraints (see \cite[Example 2]{Roc70} or \cite{BecHer22}). In particular,  our approach provides a framework to investigate linear quadratic optimal control problems with state constraints and without {coercivity assumptions in} the control variable;  {this was} pointed out in \cite{HerWol17}. 

In the context of time-continuous systems, a global \emph{characteristic method} {for Convex Bolza problems} was introduced by Rockafellar and Wolenski in \cite[Theorem 2.4]{RocWol00}.   The characteristic method describes the time-evolution of the subgradients of the value functions through following trajectories of a Hamiltonian system. The cited result holds true for Convex Bolza problems under standard hypotheses { where the Lagrangian is coercive and there are} no {state constraints} implicitly encoded (see \cite[Assumption (A) and Section 3]{RocWol00}).  {This means} the minimization in both the primal and dual problems takes places over the space of absolutely continuous arcs.  This results was later extended in \cite{HerWol19} to problem with state constraints and possibly non-coercive Lagrangian, where the minimization in both the primal and dual problems takes places, in that case, over the arcs of bounded variation.

The main contributions of this paper are the following. On the one hand, we transpose the characteristic method to discrete-time systems, by showing that the time-evolution of the subgradients of the value functions can be associated with trajectories of a discrete-time Hamiltonian system. To do so, we first need to prove that the value function turns out to be the conjugate of the value function of the dual problem. One the other hand, these results require some qualification condition on the primal as well on the dual problem. Qualification conditions over the dual problems may be rather cumbersome to verify, thus we also show that, under appropriate assumptions on the data of the primal problem, the qualification conditions over the dual problems are satisfied immediately.

\subsection{Notation and essentials}
Throughout this paper we use the following notation: $|\cdot|$ is the Euclidean norm and $a \cdot b$ stands for the Euclidean inner product of $a,b\in\bR^n$. We set $[\![p:q]\!]:=\{p,p+1,\ldots,q\}$, that is, it is the collection of all integers between $p$ and $q$ (inclusive), assuming always that $p<q$. 

Suppose $\varphi:X\to\bR\cup\{\pm\infty\}$ is a function with $X$ being a topological vector space (t.v.s. for short).  The effective domain of $\varphi$ is the set 
\[\dom(\varphi):=\{x\in X\mid\ \varphi(x)<+\infty\}.\]  
The function $\varphi$ is said to be {\it proper} if $\dom(\varphi)\neq\emptyset$ and $\varphi(x)>-\infty$ for all $x\in X$;  {\it convex} if $\epi(\varphi):=\{(x,r)\in X\times\bR\mid\ \varphi(x)\leq r\}$ is a convex set, and  {\it lower semicontinuous (l.s.c. for short)} if $\epi(\varphi)$ is a closed set.  

We also use standard nomenclature of Convex Analysis: $\ri(S)$ stands for the \emph{relative interior} of a convex set $S\subset\bR^n$. The indicator and support functions of a set $S\subset\bR^n$ are denoted by $\delta_S$ and $\sigma_S$, respectively. The set $S_\infty$ stands for the recession cone of a convex set $S\subset\bR^n$.
The (convex) normal cone to $S$ at $x\in S$ is the set
\[\mathcal{N}_S(x):=\{z\in\bR^n\mid\ z\cdot(s-x)\leq 0,\ \forall s\in S\}.\]

Suppose $Y$ is another t.v.s. which is in duality with $X$ via a bilinear mapping $\langle\cdot,\cdot\rangle:X\times Y\to\bR$. When $X=\bR^n$, we assume it is in duality with itself via the Euclidean inner product. The conjugate of $\varphi:X\to\bR\cup\{+\infty\}$ is the mapping $\varphi^*:Y\to\bR\cup\{\pm\infty\}$ defined via the formula
\[\varphi^*(y):=\sup\left\{\langle x,y\rangle-\varphi(x)\mid\ x\in X\right\},\qquad\forall y\in Y,\]
and its subdifferential  at $x\in\dom(\varphi)$ is the set
\[\partial \varphi(x):=\{y\in Y\mid \varphi(x)+\langle y,z-x\rangle \leq \varphi(z),\ \forall z\in X\}.\]
These mathematical objects are related via the Fenchel-Young equality: 
\begin{equation}\label{eq:FYeq}
	y\in\partial \varphi(x)\quad\Longleftrightarrow\quad \varphi(x)+\varphi^*(y)=\langle x,y\rangle.
\end{equation} 

A function $h:\bR^n\times\bR^n\to\bR\cup\{\pm\infty\}$ is called \emph{concave-convex} if 
$h_y(\cdot)=-h(\cdot,y)$ and $h_x(\cdot)=h(x,\cdot)$ are convex functions.
The (concave-convex) subdifferential of $h$ is the set
\begin{equation}\label{eq:subdif_saddle}
	\partial h(x,y):=[-\partial h_y(x)]\times\partial h_x(y),\qquad\forall (x,y)\in\bR^n\times\bR^n.	
\end{equation}

\section{Fully Convex Bolza problems}\label{sec:duality}
The focus of this paper will be on discrete-time dynamical systems that evolve until a given final time horizon $T>0$, which in our setting is a nonnegative integer that remains fixed all along the manuscript.

The optimization model we analyze corresponds to a Bolza problem associated with a running cost $L:[\![1:T]\!]\times\bR^n\times\bR^n\to\bR\cup\{+\infty\}$ (the Lagrangian of the problem) and a given terminal cost $g:\bR^n\to\bR\cup\{+\infty\}$. In particular, we are concerned with the value function of this discrete-time Bolza problem, which is the mapping that to any given initial position  $(\tau,\xi)\in[\![0:T-1]\times\bR^n$ assigns the optimal payoff of the problem
\begin{equation}\tag{P}\label{eq:primal}
	\vartheta_\tau(\xi):=\inf\left\{\sum_{t=\tau+1}^TL_t(x_{t-1},\Delta x_t)+g(x_T)\mid \{x_t\}_{t=\tau}^T\subset\bR^n,\ x_\tau=\xi\right\},
\end{equation}
where $\Delta x_t=x_t-x_{t-1}$ and $L_t=L(t,\cdot)$ for any $t\in [\![1:T]$. Our task in this paper is to explore duality relations for this value function, by posing the discussion in a fully convex context. The latter means that we are mainly concerned with the case when the following basic assumption holds:
\begin{align}\tag{A}\label{hyp:convex}
	\begin{cases} L_t \text{ is proper, convex and l.s.c. for any }t\in [\![1:T];\\
		g\text{ is proper, convex and l.s.c.}	
	\end{cases}
\end{align}
Unless otherwise stated, we will always assume that \eqref{hyp:convex} is in force.

\subsection{Dual Bolza problem and weak duality}

The convex setting we are studying in this paper prompts for a duality theory which can be obtained via the standard approach with perturbation functions (cf. \cite{RocWet83}). Inspired by study done in \cite{RocWol00} for time-continuous Bolza problems, we consider a dual problem to \eqref{eq:primal}, whose value function (the \emph{dual value function} in the sequel), is given, for $\tau\in [\![0:T-1]$ and $\eta\in\bR^n$, by the expression
\begin{equation}\tag{D}\label{eq:dual}
	\omega_\tau(\eta):=\inf\left\{\sum_{t=\tau+1}^TK_t(p_t,\Delta p_t)+f(p_T)\mid \{p_t\}_{t=\tau}^T\subset\bR^n,\ p_\tau=-\eta\right\}.
\end{equation}
Here $\Delta p_t=p_t-p_{t-1}$, $f:\bR^n\to\bR\cup\{+\infty\}$ is the \emph{dual terminal cost} given by \[f(b):=g^*(-b),\qquad\forall b\in\bR^n,\] 
and $K_t=K(t,\cdot)$, where $K:[\![1:T]\!]\times\bR^n\times\bR^n\to\bR\cup\{+\infty\}$ is the \emph{dual Lagrangian} which is given by
\[K(t,p,w):=\sup_{x,v\in\bR^n}\{x\cdot w+ v\cdot p-L_t(x,v)\},\qquad\forall p,w\in\bR^n.\]
Accordingly, in the sequel, we will refer to the mapping $(\tau,\xi)\mapsto \vartheta_\tau(\xi)$ given by \eqref{eq:primal} as the \emph{primal value function}.
\begin{rem}\label{rem:symmetry}
	Although problem \eqref{eq:dual} has not been written exactly as the primal problem \eqref{eq:primal}, it is not difficult to see that it can be brought to that form if we consider the Lagrangian $\tilde L_t(p,w)=K_t(p+w,w)$. In this case, it is also evident that, since $\tilde L_t^*(v,x)=K_t^*(v,x-v)=L_t(x-v,v)$, the dual problem  to \eqref{eq:dual} (defined in the way proposed above) is exactly \eqref{eq:primal}.
\end{rem}

Notice that, as for the time continuous case, we can extend the definition of both value functions up to time $\tau=T$ as follows:
\[
\vartheta_T(\xi)=g(\xi)\quad\text{and}\quad \omega_T(\eta)=f(-\eta),\qquad\forall \xi,\eta\in\bR^n.
\]
From the definition of the conjugate, it follows that
\[\vartheta_T(\xi)+\omega_T(\eta)\geq \xi\cdot\eta,\qquad\forall \xi, \eta\in\bR^n.
\]
Notice that in particular, we have that $\omega_T\geq\vartheta_T^*$ and $\vartheta_T\geq\omega_T^*$. This weak duality relation, becomes strong (with equality) whenever the transversality condition $-\eta\in \partial g(\xi)$ holds; this is a straightforward consequence of \eqref{eq:FYeq}.

The next proposition shows that this weak duality relation propagates backward in time for any $\tau\in[\![0:T]\!]$. Latter on, we will show that the weak duality stated below becomes strong under certain qualification conditions.
\begin{prop}\label{prop:weak}
	For any $\tau\in[\![0:T]\!]$ we have that 
	\[\vartheta_\tau(\xi)+\omega_\tau(\eta)\geq \xi\cdot\eta,\qquad\forall \xi, \eta\in\bR^n.\]
	Moreover, we also have that $\omega_\tau\geq\vartheta_\tau^*$ and $\vartheta_\tau\geq\omega_\tau^*$.
\end{prop}

\begin{proof}
	We note that it is enough to check the inequality for $\xi,\eta\in\bR^n$ such that $\vartheta_\tau(\xi), \omega_\tau(\eta)<+\infty$, otherwise the conclusion is straightforward (considering the convention $+\infty\pm\infty=+\infty$). Note that we are not assuming a priori that $\vartheta_\tau(\xi),\omega_\tau(\eta)>-\infty$. 
	
	Let $\{x_t\}_{t=\tau}^T\subset\bR^n$ and $\{p_t\}_{t=\tau}^T\subset\bR^n$ be  feasible processes for the optimization problems defined in \eqref{eq:primal} and \eqref{eq:dual}, respectively. Notice that by the definition of the conjugate we have
	$$L_t(x_{t-1},\Delta x_t)+K_t(p_t,\Delta p_t)
	\geq x_{t-1}\cdot\Delta p_t+\Delta x_t\cdot p_t=x_t\cdot p_t-x_{t-1}\cdot p_{t-1}.$$
	Thus,
	$$\sum_{t=\tau+1}^TL_t(x_{t-1},\Delta x_t)+\sum_{t=\tau+1}^TK_t(p_t,\Delta p_t)\geq x_T\cdot p_T-x_\tau\cdot p_\tau=x_T\cdot p_T+\xi\cdot \eta.$$
	Since $f(b)=g^*(-b)$, the definition of the conjugate of $g$ leads to
	\[g(x_T)+f(p_T)\geq-x_T\cdot p_T.\]
	It follows that
	\[\sum_{t=\tau+1}^TL_t(x_{t-1},\Delta x_t)+g(x_T)+ \sum_{t=\tau+1}^TK_t(p_t,\Delta p_t)+f(p_T)\geq\xi\cdot\eta. \]
	Taking infimum over  $\{x_t\}_{t=\tau}^T\subset\bR^n$ and $\{p_t\}_{t=\tau}^T\subset\bR^n$ we get the desired inequality. In particular, we must have that $\vartheta_\tau(\xi)>-\infty$ and $\omega_\tau(\eta)>-\infty$, and thus $\eta\in\dom(\omega_\tau)$ and $\xi\in\dom(\vartheta_\tau)$.
\end{proof}

\subsection{Strong duality}
Let us now focus on proving a strong duality result, which in our case means that the primal and dual value functions are conjugate to each other. For such purpose we introduce the following qualification condition:
\begin{align}\tag{H}\label{hyp:qual_prim}
	\begin{cases}
		\exists \{\bar x_t\}_{t=0}^T\subset\bR^n,\quad\text{such that }\bar x_T\in\ri\left(\dom(g)\right),\\
		\bar x_{t-1}\in\ri\left(\X(t)\right)\quad\text{and}\quad\Delta \bar x_t\in \ri\left(\Gamma_{L}(t,\bar x_{t-1})\right),\quad\forall t\in[\![1:T]\!].
	\end{cases}
\end{align}
where
\[\Gamma_{L}(t,x):=\{v\in\bR^n\mid\ L_t(x,v)\in\bR\},\,\quad\text{and}\quad\,\X(t):=\{x\in\bR^n\mid\ \Gamma_{L}(t,x)\neq\emptyset\}.\]

\begin{rem}
	Notice that the minimization process in  \eqref{eq:primal} can be restrained to processes whose initial position $x_\tau=\xi$ is brought to the target $\dom(g)$ at time $t=T$. This means that the set $\dom(g)$ can in this context be understood as a terminal constraint, implicitly  encoded in the formulation of the problem. 
	
	Similarly, by allowing each $L_t$ to take infinite values, we are handling implicitly constraints over the state of system $x_t$ and the variation $\Delta x_t$. Indeed, any feasible processes of the Bolza problem  \eqref{eq:primal} must satisfy:
	\begin{eqnarray*}
		x_{t-1}\in\X(t)\quad\text{and}\quad\Delta x_t\in \Gamma_{L}(t,x_{t-1}),\quad\forall t\in[\![\tau+1:T]\!].
	\end{eqnarray*}
	In other words, the set-valued maps $\Gamma_{L}$ and $\X$ correspond respectively to the underlying dynamics of the system and  to the (time-dependent) state-constraint.
	
	Taking this into account, it follows that the qualification condition \eqref{hyp:qual_prim} can be understood as a \emph{strict feasibility assumption} over the dynamical system.
\end{rem}

We are now in a position to establish a strong duality relation between the primal and dual value functions.

\begin{prop}\label{prop:conjugacy_1}
	Under Hypothesis \eqref{hyp:qual_prim}, for any $\tau\in[\![0:T]\!]$ it follows that $\vartheta_\tau^*=\omega_\tau$.
	Moreover, the infimum in the definition of $\omega_\tau(\eta)$ is attained for any $\eta\in\dom(\omega_\tau)$.
\end{prop}
\begin{proof}
	Let $\tau\in[\![0:T]\!]$ be given. Note that from \eqref{hyp:qual_prim}, we get $\vartheta_\tau (\bar x_\tau)<+\infty$. In particular, it follows that
	\[\vartheta_\tau^*(\eta)\geq \bar x_\tau\cdot \eta- \vartheta_\tau (\bar x_\tau)>-\infty,\qquad\forall \eta\in\bR^n.\]
	
	Let us focus on proving first $\omega_\tau=\vartheta_\tau^*$.  Let us define $\mathbf E_\tau=\prod_{t=\tau}^T\bR^n$, and observe that for $\eta\in\dom(\vartheta_\tau^*)$ we have
	$$-\vartheta^*_\tau(\eta)=\inf_{\xi\in\bR^n}\left\{\vartheta_\tau(\xi)-\xi\cdot\eta\right\}=\inf_{\mathbf{x}\in\mathbf E_\tau}\left\{\sum_{t=\tau+1}^TL_t(x_{t-1},\Delta x_t)+g(x_T)-x_\tau\cdot\eta\right\},$$
	where we have considered $\mathbf{x}=\{x_t\}_{t=\tau}^T$. If we define $\ell(a,b)=g(b)-a\cdot\eta$, it is not difficult to see that
	$\ri(\dom(\ell))=\bR^n\times \ri(\dom(g))$, and so, from the arguments used to prove \cite[Theorem 2]{RocWet83}, it follows that \eqref{hyp:qual_prim} implies that $\partial\phi(0)\neq\emptyset$, where for any $\mathbf{y}=\{y_t\}_{t=\tau}^T\in\mathbf E_\tau$ the perturbation function $\phi$ is
	\begin{equation}\label{eq:perturbation}
		\phi(\mathbf{y})=\inf_{\mathbf{x}\in \mathbf E_\tau}\left\{\sum_{t=\tau+1}^TL_t(x_{t-1},\Delta x_t+y_t)+\ell(x_\tau+y_\tau,x_T)\right\}.
	\end{equation}
	From \cite[Theorem 3]{RocWet83} we get that
	$$\vartheta^*_\tau(\eta)=\min_{\mathbf{p}\in\mathbf E_\tau}\left\{\sum_{t=\tau+1}^TK_t(p_t,\Delta p_t)+\ell^*(p_\tau,-p_T)\mid \mathbf{p}=\{p_t\}_{t=\tau}^T\right\},$$
	and the minimum is attained at some $\{p^*_t\}_{t=\tau}^T\in\partial\phi(0)\subset\mathbf E_\tau$.
	Therefore, since $\ell^*(a,b)=g^*(b)$ if $a=-\eta$ and $\ell^*(a,b)=+\infty$ otherwise, we get that $\omega_\tau=\vartheta_\tau^*$; for the case $\vartheta_\tau^*(\eta)=+\infty$ the equality  follows from the weak duality (Proposition \ref{prop:weak}).
\end{proof}

In a symmetric way, if now a qualification condition is imposed over the data of the dual problem, a similar result can be obtained. To be more precise, let us consider the following qualification condition

\begin{align}\tag{H'}\label{hyp:qual_dual}
	\begin{cases}
		\exists \{\bar p_t\}_{t=0}^T\subset\bR^n,\quad\text{such that }\bar p_T\in\ri\left(\dom(f)\right),\\
		\bar p_{t-1}\in\ri\left(\P(t)\right)\quad\text{and}\quad\Delta \bar p_t\in \ri\left(\Gamma_{K}(t,\bar p_{t-1})\right),\quad\forall t\in[\![1:T]\!],
	\end{cases}
\end{align}
where,
\[\Gamma_K(t,p):=\{w\in\bR^n\mid\ K_t(p+w,w)\in\bR\},\,\quad\text{and}\quad \P(t):=\{p\in\bR^n\mid\ \Gamma_K(t,p)\neq\emptyset\}.\] 
\begin{rem}
	Similarly as for the primal problem, the set-valued maps $\Gamma_{K}$ and $\P$ can be interpreted respectively as the  dynamics and as the state-constraint of the \emph{dual system}, which are implicitly encoded in the formulation of the problem. Also, $\dom(f)$ corresponds to the underlying terminal constraint of the problem. Consequently,  any feasible arc of the dual Bolza problem \eqref{eq:dual} must satisfy the following constraints:
	\begin{eqnarray*}
		p_{t-1}\in\P(t)\quad\text{and}\quad\Delta p_t\in \Gamma_K(t,p_{t-1}),\quad\forall t\in[\![\tau+1:T]\!]\quad\text{and}\quad p_T\in\dom(f).
	\end{eqnarray*}

	In this context, the qualification condition \eqref{hyp:qual_dual} can be understood as a \emph{strict feasibility assumption} over the dual dynamical system.	
\end{rem}

By symmetry, it is not difficult to see (in the light of Remark \ref{rem:symmetry}), that the following statement holds true.

\begin{prop}\label{prop:conjugacy_2}
	Under Hypothesis \eqref{hyp:qual_dual}, for any $\tau\in[\![0:T]\!]$ it follows that $\omega_\tau^*=\vartheta_\tau$.
	Moreover, the infimum in the definition of $\vartheta_\tau(\xi)$ is attained for any $\xi\in\dom(\vartheta_\tau)$.
\end{prop}

As a consequence, under Hypotheses \eqref{hyp:qual_prim} and \eqref{hyp:qual_dual}, for any $\tau\in[\![0:T]\!]$ we have that $\vartheta_\tau$ and $\omega_\tau$ are convex proper and l.s.c. functions on $\bR^n$. Moreover, they are conjugate to each other, that is,
\[\vartheta_\tau^*=\omega_\tau\quad\text{and}\quad  \vartheta_\tau=\omega_\tau^*.\]

\section{Discrete-time Characteristic method}
We now present and prove a discrete-time characteristic method, whose main goal is to provide a description of the evolution of the subgradients of the primal value function by means of a discrete-time Hamiltonian system.

Let $H:[\![0:T]\!]\times\bR^n\times\bR^n\to\bR\cup\{\pm\infty\}$ be the Hamiltonian associated with the primal problem \eqref{eq:primal}, that is
\[H(t,x,p):=\sup_{v\in\bR^n}\left\{p\cdot v-L_t(x,v)\right\}.\]
For $t\in [\![1:T]\!]$ fixed, this function is concave in $x$ and convex in $p$, thus its subdifferential is given by \eqref{eq:subdif_saddle}. In particular, setting $H_t=H(t,\cdot)$, we have that $(-w,v)\in \partial H_t(\bar x,\bar p)$ if and only if
\[H_t(x,\bar p)+w\cdot(x-\bar x)\leq H_t(\bar x,\bar p)\leq H_t(\bar x,p)-v\cdot(p-\bar p),\qquad\forall x,p\in\bR^n.\]

\begin{dfn}\label{dfn:Hamiltonian}
	For any $\tau\in[\![0:T-1]\!]$, we say that $\{(x_t,p_t)\}_{t=\tau}^T\subset\bR^n\times\bR^n$ is a \emph{discrete-time Hamiltonian trajectory} on $[\![\tau:T]\!]$ provided that 
	\begin{equation}\label{eq:dHt}
		(-\Delta p_t,\Delta x_t)\in\partial H_t(x_{t-1},p_t),\qquad\forall t\in[\![\tau+1,T]\!].
	\end{equation}
\end{dfn}

The characteristic method and main result of this paper is the following.
\begin{thm}
	\label{thm:characteristics}
	Let  $\tau\in[\![0:T-1]\!]$ and $\xi,\eta\in\bR^n$ be given. Suppose that there is a  \emph{discrete-time Hamiltonian trajectory} on $[\![\tau:T]\!]$, say  $\{(x_t,p_t)\}_{t=\tau}^T\subset\bR^n\times\bR^n$, such that $(x_\tau,p_\tau)=(\xi,-\eta)$ and that satisfies the transversality condition \[-p_T\in\partial g(x_T).\]
	Then, $\eta\in\partial\vartheta_\tau(\xi)$. Moreover, the converse holds true if the Hypotheses \eqref{hyp:qual_prim} and \eqref{hyp:qual_dual} are in force.
\end{thm}
\begin{proof}
	Let $l:\bR^n\times\bR^n\to\bR$ be the convex proper and l.s.c. function given by
	\[l(a,b)=g(b)-a\cdot\eta\]
	Let $\{(x_t,p_t)\}_{t=\tau}^T\subset\bR^n\times\bR^n$ be given. It is not difficult to see that the transversality condition $-p_T\in\partial g(x_T)$ combined with $p_\tau=-\eta$ is equivalent to
	\begin{equation}\label{eq:tc}
		(p_\tau,-p_T)\in\partial\ell(x_\tau,x_T).	
	\end{equation}
	Moreover, since 
	\[H(t,x,p):=\sup_{v\in\bR^n}\left\{p\cdot v-L_t(x,v)\right\},\]
	from \cite[Theorem 37.5]{RocBook70}, we have that the condition of being a \emph{discrete-time Hamiltonian trajectory} on $[\![\tau:T]\!]$ (see equation \eqref{eq:dHt}) is equivalent to the \emph{discrete-time Euler-Lagrange relation}
	\begin{equation}\label{eq:dEUr}
		(\Delta p_t,p_t)\in\partial L_t(x_{t-1},\Delta x_t),\qquad\forall t\in [\![\tau+1:T]\!].
	\end{equation}
	Let $\{(x_t,p_t)\}_{t=\tau}^T$ be as on the statement of the theorem. From \cite[Theorem 1]{RocWet83}, we have that  $\{x_t\}_{t=\tau}^T$
	is an optimal trajectory for
	\begin{equation}\label{eq:auxp}\tag{$\mathcal{P}_0$}
		\inf_{\mathbf{y}\in\mathbf E_\tau}\left\{\sum_{t=\tau+1}^TL_t(y_{t-1},\Delta y_t)+g(y_T)-y_\tau\cdot\eta\right\}=\inf_{y_\tau\in\bR^n}\left\{\vartheta_\tau(y_\tau)-y_\tau\cdot\eta\right\},
	\end{equation}
	and that $\{p_t\}_{t=\tau}^T\in\partial\phi(0)$, where we recall that $\phi$ is the perturbation function given by \eqref{eq:perturbation}. From \cite[Theorem 3]{RocWet83}, it follows that $\{p_t\}_{t=\tau}^T\subset\bR^n$ realizes $\omega_\tau(\eta)$, and moreover, $\mathop{\rm val}\eqref{eq:auxp}=-\omega_\tau(\eta)$.
	Notice that  $\{x_t\}_{t=\tau}^T$ realizes  $\vartheta_\tau(\xi)$ as well, and that $\vartheta_\tau(\xi)=\mathop{\rm val}\eqref{eq:auxp}+\xi\cdot\eta$. From Proposition \ref{prop:weak}, we get that $\vartheta_\tau$ and $\omega_\tau$ are conjugate to each other, and that the Fenchel-Young equality holds at $(\xi,\eta)$, and so $\eta\in\partial\vartheta_\tau(\xi)$.
	
	Conversely, let us assume that \eqref{hyp:qual_prim} and \eqref{hyp:qual_dual} hold and that $\eta\in\partial\vartheta_\tau(\xi)$. In particular $\xi\in\dom(\vartheta_\tau)$, and so, by Proposition \ref{prop:conjugacy_2}, there is an optimal trajectory that realizes $\vartheta_\tau(\xi)$, that is, there is $\{x_t\}_{t=\tau}^T\subset\bR^n$ such that $x_\tau=\xi$ and
	\[\vartheta_\tau(\xi)=\sum_{t=\tau+1}^TL_t(x_{t-1},\Delta x_t)+g(x_T).\]
	
	Notice that the condition $\eta\in\partial\vartheta_\tau(\xi)$ is equivalent to $\xi\in\partial\omega_\tau(\eta)$ because $\vartheta_\tau$ and $\omega_\tau$ are conjugate to each other. Hence, $\eta\in\dom(\omega_\tau)$, and thus Proposition \ref{prop:conjugacy_1} implies that there is an optimal trajectory that realizes $\omega_\tau(\eta)$, that is, $\{p_t\}_{t=\tau}^T\subset\bR^n$ such that $p_\tau=-\eta$ and 
	$$\omega_\tau(\eta)=\sum_{t=\tau+1}^TK_t(p_t,\Delta p_t)+f(p_T).$$
	
	From the Fenchel-Young equality we have that $\vartheta_\tau(\xi)+\omega_\tau(\eta)=\xi\cdot\eta$. Thus $\{x_t\}_{t=\tau}^T\subset\bR^n$ realizes the infimum in \eqref{eq:auxp}. Moreover, from the proof of Proposition \ref{prop:conjugacy_1} we also have that $\{p_t\}_{t=\tau}^T\in\partial\phi(0)$. It follows from \cite[Theorem 1]{RocWet83} that  the transversality condition \eqref{eq:tc} and  the discrete-time Euler-Lagrange relation \eqref{eq:dEUr} holds.
	We have already discussed that these conditions are equivalent to the transversality condition given on the statement and to be a discrete-time Hamiltonian trajectory, so the conclusion follows.
\end{proof}

\section{Examples}
Let us now provide some explicit formulae for two type of optimal control problems that fit into the convex setting we have posed for our analysis. The dynamics that govern both problems are jointly linear in the state and in the control. In particular, throughout this section,  we consider some matrices $A_0,\ldots,A_{T-1}$ and $B_0,\ldots,B_{T-1}$ of dimension  ${n\times n}$ and ${n\times m}$, respectively, and some  given \emph{drifts} $\varphi_0,\ldots,\varphi_{T-1}\in\mathbb{R}^n$; these serve as the building blocks for the linear systems we discuss below.

We also take $\mathcal{X}_{0},\ldots,\mathcal{X}_T\subset\mathbb{R}^n$ and $\mathcal{U}_0\ldots\mathcal{U}_{T-1}\subset\mathbb{R}^m$ to be some given convex, closed, and nonempty sets; the first ones will play the role of (pure) state constraints and the second ones of (pure) control constraints.

\subsection{Linear Quadratic problem with state constraints}
The first example we study is the so-called the linear quadratic (LQ) problem. We show that this model can be studied in the framework of this paper.  Throughout this example, we consider the positive semidefinite matrices $Q_{1},\ldots,Q_T$ and $R_0,\ldots,R_{T-1}$ of dimension ${n\times n}$ and  ${m\times m}$, respectively, associated with the objective cost function of the problem; we set $Q_0$ to be the zero matrix of dimension ${n\times n}$.

Given an initial position $(\tau,\xi)\in[\![0:T-1]\!]\times\mathbb{R}^n$, the (LQ) problem we are interested in can be stated as follows:
\begin{equation}\label{eq:LQ}\tag{LQ}
	\begin{cases}
		\text{ Minimize }&\dfrac{1}{2}\left[\displaystyle\sum_{t=\tau}^{T-1}\left(\|x_{t+1}\|^2_{Q_{t+1}}+\|u_{t}\|^2_{R_t}\right)\right]\\
		\text{ over all }& \{x_t\}_{t=\tau}^T\subset\bR^n\text{ and }\{u_t\}_{t=\tau}^{T-1}\subset\bR^n\text{ with }x_\tau=\xi,\\
		\text{ such that }& \Delta x_{t+1}= A_tx_{t}+B_tu_{t}+\varphi_t,\qquad\forall t\in[\![\tau:T-1]\!],\\
		&u_t\in \mathcal{U}_t,\qquad\forall t\in[\![\tau:T-1]\!],\\
		&x_t\in \mathcal{X}_t,\qquad\forall t\in[\![\tau:T]\!].
	\end{cases}	
\end{equation}
Here we use the notation $\|z\|^2_{M}:=z\cdot M z$ for $z\in\mathbb{R}^d$, and for a $d\times d$ matrix $M$.  

To formulate the problem as in \eqref{eq:primal}, we set the endpoint cost $g:\mathbb{R}^n\to\mathbb{R}$ as $g(a):=\frac{1}{2}\|a\|^2_{Q_T}$, and the Lagrangian $L:[\![1:T]\!]\times\bR^{n}\times\mathbb{R}^{n}\to\bR\cup\{+\infty\}$ as
\[L(t+1,x,v)=\frac{1}{2}\|x\|^2_{Q_t}+\delta_{\mathcal{X}_t}(x)+\inf_{u\in{\mathcal{U}_{t}}}\left\{\frac{1}{2}\|u\|^2_{R_t}\mid v=A_tx+B_tu+\varphi_t\right\}.\]

In order to ensure the l.s.c of the Lagrangian, we assume the following qualification condition.
\begin{align}\label{hyp:lq_calif}\tag{CQ}
	\ker(B_t)\cap\ker(R_t)\cap(\mathcal{U}_t)_\infty=\{0\},\qquad\forall t\in[\![0:T-1]\!].
\end{align}
Given that the recession function of $u\mapsto \frac{1}{2}\|u\|^2_{R_t}+\delta_{\mathcal{U}_t}(u)$ is the indicator function of $\ker(R_t)\cap(\mathcal{U}_t)_\infty$, we can conclude that \eqref{hyp:lq_calif} implies the qualification condition as per \cite[Corollary 3.5.7]{AusTebBook03}.  Consequently, the infimum in the Lagrangian's definition is attained, rendering it an l.s.c. function. 

Set  $$S_t:=\{(x,u,v)\in \bR^n\times \bR^m\times \bR^n \mid v=A_tx+B_tu+\varphi_t\}.$$
We can easily see that we have:
\begin{align*}
	&\vartheta(\tau,\xi)=\inf_{\substack{\{x_t\}\subset \R^n \\x_\tau = \xi}}  \sum_{t=\tau+1}^{T} \left[ L_t(x_{t-1},\Delta x_t) \right] + g(x_T) \\
	&=  \inf_{\substack{\{x_t\}\subset \R^n \\x_\tau = \xi}} \frac{1}{2} \sum_{t=\tau}^{T-1} \left[ \|x_{t}\|_{Q_{t}}^2 +  \delta_{\mathcal{X}_t}(x_t)
	+ \inf_{u_t\in \mathcal{U}_t}\{ \|u_t\|_{R_t}^2  + \delta_{S_t}(x_t,u_t,\Delta x_{t+1}) \} \right] + \frac{1}{2}\|x_T\|_{Q_T}^2\\
	& =  \inf_{\substack{\{x_t\}\subset \R^n \\ \{u_t\}\subset \R^m\\x_\tau = \xi}} \frac{1}{2} \sum_{t=\tau}^{T-1}  \left[ \|x_{t+1}\|_{Q_{t+1}}^2 + \|u_t\|_{R_t}^2 + \delta_{\mathcal{X}_t}(x_t)+\delta_{\mathcal{U}_t}(u_t)
	+ \delta_{S_t}(x_t,u_t,\Delta x_{t+1}) \right].
\end{align*}
Therefore, the Bolza formulation is equivalent to the original problem. Turning our attention to the dual problem, we proceed to calculate both $f$ and $K_t$.
\[
f(b)=g^{\ast}(-b) = \sup_{x\in \bR^n}\left\{- x\cdot b -\frac{1}{2} \|x\|_{Q_T}^2 \right\}.
\]
When $Q_T$ is positive definite, the supremum is achieved at $x=-Q_T^{-1}b$, and consequently, $f(b)= \frac{1}{2}\|b\|_{Q_T^{-1}}^2$. On the other hand
\begin{equation*}
	\begin{split}
		&K(t+1,p,w)  = \sup_{x,v\in \R^n} \left\{  x\cdot w  +  v\cdot p  -L_{t+1}(x,v)\right\} \\
		& = \sup_{x,v\in \R^n} \left\{  x\cdot w  +  v\cdot p  -\frac{1}{2}\left(\|x\|_{Q_{t}}^2 +  \delta_{\mathcal{X}_t}(x)+ \inf_{u\in \mathcal{U}_t}\{ \|u\|_{R_t}^2 + \delta_{S_t}(x,u,v) \}\right)\right\} \\
		& = \sup_{\substack{x\in\mathcal{X}_t \\ u\in\mathcal{U}_t \\ v\in \R^n}}\left\{   x\cdot w  +  v\cdot p  -\frac{1}{2}\|x\|_{Q_t}^2 - \frac{1}{2}\|u\|_{R_t}^2 \mid  v=A_tx+B_tu+\varphi_t \right\}\\
		& = \sup_{\substack{x\in\mathcal{X}_t \\ u\in\mathcal{U}_t}} \left\{  x\cdot (A_t^Tp+w) -\frac{1}{2}\|x\|_{Q_t}^2 +  u\cdot B_t^Tp  -\frac{1}{2}\|u\|_{R_t}^2 \right\} + \varphi_t\cdot p ,\\
		& = \sup_{x\in\mathcal{X}_t} \left\{ x\cdot (A_t^Tp+w) -\frac{1}{2}\|x\|_{Q_t}^2\right\} + \sup_{u\in\mathcal{U}_t}\left\{ u\cdot B_t^Tp  -\frac{1}{2}\|u\|_{R_t}^2 \right\} + \varphi_t\cdot p .
	\end{split} 
\end{equation*}

We now focus on providing conditions on the primal problem to ensure \eqref{hyp:qual_prim} and \eqref{hyp:qual_dual}. Recall that $\Gamma_L(t,x) =  \{v\in \bR^n : L_t(x,v)\in \bR \}$, then
\[\Gamma_L(t,x)  \left\{v\in \bR^n\mid \|x\|_{Q_{t-1}}^2 +  \delta_{\mathcal{X}_{t-1}}(x)\\
+ \inf_{u\in \mathcal{U}_{t-1}}\left\{ \|u\|_{R_{t-1}}^2 + \delta_{S_{t-1}}(x,u,v) \right\} \in \bR\right\}.
\]
From this, it is staightforward that if $x\notin\mathcal{X}_{t-1}$, then $\Gamma_L(t,x)=\emptyset$. If $x\in\mathcal{X}_{t-1}$, then  $v\in \Gamma_L(t,x)$ if and only if  $\inf\left\{ \|u\|_{R_{t-1}}^2 + \delta_{S_{t-1}}(x,u,v)\mid u\in \mathcal{U}_{t-1} \right\} \in \bR$. Given that the function $\|\cdot\|_{R_t}^2$ is real valued, and bounded from below, it follows that if $x\in\mathcal{X}_{t-1}$, and
\begin{equation*}
	\begin{split}
		v\in \Gamma_L(t,x) & \Leftrightarrow \exists u\in\mathcal{U}_{t-1} \ \textrm{s.t.} \ v=A_{t-1}x+B_{t-1}u+\varphi_{t-1}, \\
		& \Leftrightarrow v\in B_{t-1}\mathcal{U}_{t-1} + A_{t-1}x+ \varphi_{t-1}.
	\end{split}
\end{equation*}
Therefore, if $x\in\mathcal{X}_{t-1}$, $\Gamma_L(t,x)=B_{t-1}\mathcal{U}_{t-1}+A_{t-1}x+\varphi_{t-1}$.  As $\mathcal{U}_{t-1}$ is nonempty, we have that $\X(t)=\mathcal{X}_{t-1}$ for all $t$.  Thus, if $\mathcal{X}_t=\bR^n$ for every $t$, \eqref{hyp:qual_prim} holds true due to the fact that
\begin{equation*}
	\ri (\Gamma_L(t,x)) =  \ri(B_{t-1}(\mathcal{U}_{t-1}))+A_{t-1}x+\varphi_{t-1} = B_{t-1}(\ri (\mathcal{U}_{t-1}))+A_{t-1}x+\varphi_{t-1}
\end{equation*}
is always nonempty as $\mathcal{U}_{t-1}$ is convex.

We now study the case where $\mathcal{X}_1=\ldots=\mathcal{X}_T=X$. We look for conditions to have an equilibrium point in the system, that is, $x\in \ri(X)$ such that $0\in B_t(\ri(\mathcal{U}_t))+A_tx+\varphi_t$. This condition implies \eqref{hyp:qual_prim}, as we can choose $\bar{x}_1=\ldots=\bar{x}_T=x$. To verify the existence of such a point, we can solve the following linear system of equations
\begin{equation*}
	\begin{bmatrix}
		A_t  & B_t 
	\end{bmatrix} \begin{bmatrix}
		x \\ u
	\end{bmatrix} = -\varphi_t
\end{equation*}
for each $t$, and intersect the solution sets with $\ri(X)\times \ri(\mathcal{U}_t)$. If there is some $\bar{x}$ that belongs to this intersection for every $t$, then we have \eqref{hyp:qual_prim}. In particular, we can check if  $0\in \ri(X)$ and $0\in B_t(\ri(\mathcal{U}_t))+\varphi_t$. If $\varphi_t=0$, we can replace it with simply $0\in \ri(X)$ and $0\in\ri(\mathcal{U}_t)$. Moreover, if $X$ has a nonempty interior and $A_t$ has rank $n$, then it is enough to check if $0\in\mathop{\rm int}(X)$ and $0\in\mathcal{U}_t$. This is due to the fact that for $\varepsilon>0$ we can take  $\xi\in \R^m$ with $|\xi|$ a small enough such that $\xi\in \ri(\mathcal{U}_t)$, and as $A_t$ has rank $n$ there exists $\tilde{x}\in \mathbb{B}(0,\varepsilon)$ such that $B_t\xi=-A_t(\tilde{x})$.

In the general case, by the same argument as before, we have that if the intersection between the solution set of the following system 
\begin{equation}\label{eq_sys}
	\begin{split}
		\begin{bmatrix}
			A_1+I_n & -I_n & B_1 
		\end{bmatrix} \begin{bmatrix}
			x_1 \\ x_2 \\ u_1
		\end{bmatrix} & = -\varphi_1, \\
		\begin{bmatrix}
			A_2+I_n & -I_n  & B_2 
		\end{bmatrix} \begin{bmatrix}
			x_2 \\ x_3 \\ u_2
		\end{bmatrix} & = -\varphi_2, \\
		& \vdots \\
		\begin{bmatrix}
			A_{T-1}+I_n  & -I_n & B_{T-1} 
		\end{bmatrix} \begin{bmatrix}
			x_{T-1} \\ x_T \\ u_{T-1}
		\end{bmatrix} & = -\varphi_{T-1},            
	\end{split}
\end{equation}
and the set $\ri(\mathcal{X}_1)\times\ri(\mathcal{U}_1)\times...\times\ri(\mathcal{U}_{T-1})\times\ri(\mathcal{X}_T)$ is nonempty, then condition \eqref{hyp:qual_prim} holds. In particular, if the following conditions are satisfied: $0\in \ri(\mathcal{X}_t)$, $0\in \ri(\mathcal{U}_t)$ for every $t$, $0\in {\rm int}(\mathcal{X}_t)$ ,$0\in\mathcal{U}_t$, and $A_t$ has rank $n$ for every $t$,  then condition \eqref{hyp:qual_prim} remains true.

In terms of the dual problem, to find conditions for \eqref{hyp:qual_dual}, recall that $\Gamma_K(t+1,p) = \{w\in \bR^n\mid  K_t(p+w,w)\in \bR\}$. Then, $w\in\Gamma_K(t+1,p)$ if and only if
\[ \sup_{x\in\mathcal{X}_t}\{x\cdot(A_t^Tp+A_t^Tw+w) -\frac{1}{2}\|x\|_{Q_t}^2\} + \sup_{u\in\mathcal{U}_t}\{u\cdot (B_t^Tp+B_t^Tw) -\frac{1}{2}\|u\|_{R_t}^2\} + \varphi_t\cdot (p+w) \in \bR.
\]
As the functions $(x,w)\mapsto x\cdot(A_t^Tp+A_t^Tw+w)$ and $(u,w)\mapsto u\cdot (B_t^Tp+B_t^Tw)$ are linear in $x$ and $u$ respectively, it follows that
\begin{equation*}
	\begin{split}
		w \in \Gamma_K(t+1,p) & \Leftrightarrow \forall x \in (\mathcal{X}_t)_{\infty}\cap \ker(Q_t) ,\quad   x\cdot(A_t^Tp+A_t^Tw+w) \leq 0, \\
		& \quad \wedge  \forall u \in (\mathcal{U}_t)_{\infty}\cap \ker(R_t),\quad   u\cdot (B_t^Tp+B_t^Tw) \leq 0.
	\end{split}
\end{equation*}
Note that $\Gamma_K(t+1,p)$ is convex, which implies that $\ri(\Gamma_K(t+1,p))$ is nonempty. Therefore:
\begin{itemize}
	\item If $Q_t$ is definite positive or $\mathcal{X}_t$ is compact for every $t$, then given $p\in \bRn$ the first condition holds true for every $w\in \bRn$. We can choose $w=-p$ to ensure that the second condition also remains true. Thus, we conclude that in the case $\mathbf{P}(t)=\bRn$ for every $t$, the condition \eqref{hyp:qual_dual} holds true.
	\item If $R_t$ is definite positive or $\mathcal{U}_t$ is compact for every $t$, then given $p\in \bRn$ the second condition is true for every $w\in \bRn$. If we also have that the matrix $A_t^T+I_n$ has rank $n$, then we can choose $w$ such that  $(A_t^T+I_n)w=-A_t^Tp$ and so $\mathbf{P}(t)=\bRn$ for every $t$, and \eqref{hyp:qual_dual} holds in this case.
\end{itemize}
Furthermore, if $\mathcal{X}_t$ and $\mathcal{U}_t$ are cones, then $(\mathcal{X}_t)_{\infty}=\mathcal{X}_t$ and $(\mathcal{U}_t)_{\infty}=\mathcal{U}_t$, simplifying the conditions for easier analysis and computations.

Lastly, we focus on the unrestricted case with $\mathcal{X}_t=\bRn$ and $\mathcal{U}_t=\bR^m$ for every $t$. Given $p\in \bRn$, we have that $w\in  \Gamma_K(t+1,p)$ if and only if for any $x \in  \ker(Q_t)$ and any $u \in  \ker(R_t)$ we have
\[ x\cdot(A_t^Tp+A_t^Tw+w) \leq 0\quad\text{and}\quad
\quad   u\cdot (B_t^Tp+B_t^Tw) \leq 0.
\]
In this case, $\mathbf{P}(t)$ is a vector space. Indeed, let $p_1,p_2\in \mathbf{P}(t)$, and $w_1\in  \Gamma_K(t+1,p_1)$, $ w_2\in \Gamma_K(t+1,p_2)$. Taking $w=w_1+w_2\in \bRn$, it is clear that 
\begin{gather*}
	\forall x \in \ker(Q_t) ,\quad  x\cdot(A_t^T(p_1+p_2)+A_t^T(w_1+w_2)+w_1+w_2) \leq 0,\\
	\forall u \in \ker(R_t),\quad   u\cdot (B_t^T(p_1+p_2)+B_t^T(w_1+w_2)) \leq 0,
\end{gather*}
and thus $p_1+p_2\in \mathbf{P}(t)$. Let now $p\in \mathbf{P}(t)$, $w\in  \Gamma_K(t+1,p)$ and $\lambda\in \bR$. If $\lambda \geq 0$, it is evident that $\lambda p\in \mathbf{P}(t)$ as $\lambda w\in  \Gamma_K(t+1,\lambda p)$. If $\lambda<0$, we can take $\lambda w$ and obtain
\begin{gather*}
	\forall x \in  \ker(Q_t) ,\quad   -x\cdot (A_t^T(\lambda p)+A_t^T(\lambda w)+\lambda w) \leq 0, \\
	\forall u \in  \ker(R_t),\quad   -u\cdot (B_t^T(\lambda p)+B_t^T(\lambda w))\leq 0.
\end{gather*}
As $\ker(Q_t)$ and $\ker(R_t)$ are vector spaces, $\lambda w\in  \Gamma_K(t+1,\lambda p)$ and therefore $\lambda p\in \mathbf{P}(t)$.  Consequently, $\mathbf{P}(t)$ is a vector space. This implies, in particular, that $0\in \ri(\mathbf{P}(t))$ for every $t$.  We take $\bar{p}_1=\ldots=\bar{p}_T=0$ in \eqref{hyp:qual_dual}, as it's easy to see that the set $ \Gamma_K(t+1,0)$ is a vector space too, and therefore $0\in \ri( \Gamma_K(t+1,0))$. In fact, let $w_1,w_2\in  \Gamma_K(t+1,0)$ and $\lambda\in \bR$. Then
\begin{gather*}
	\forall x \in \ker(Q_t) ,\quad  x\cdot (A_t^T(w_1+w_2)+w_1+w_2) \leq 0, \\
	\forall u \in \ker(R_t),\quad   u\cdot (B_t^T(w_1+w_2)) \leq 0,
\end{gather*}
and
\begin{gather*}
	\forall x \in  \ker(Q_t) ,\quad  \langle x,A_t^T(\lambda w)+\lambda w\rangle \leq 0, \\
	\forall u \in  \ker(R_t),\quad  \langle u,B_t^T(\lambda w)\rangle \leq 0,
\end{gather*}
if $\lambda\geq 0$, and we can use the same argument as before for $\lambda<0$.
Therefore, in the unrestricted case \eqref{hyp:qual_dual} holds.  

We summarize the main results of this subsection in the following proposition:
\begin{prop}
	The problem \eqref{eq:LQ} admits a Bolza formulation under \eqref{hyp:lq_calif} and therefore Theorem \ref{thm:characteristics} holds true. Moreover, the converse also holds true if one of the following conditions holds
	\begin{enumerate}
		\item $\mathcal{X}_t=\bRn$, for every $t\in[\![1:T]\!]$ and either 
		\begin{enumerate}
			\item $\mathcal{U}_t=\bR^m$ for all $t\in[\![1:T]\!]$,
			\item $Q_t$ is positive definite for all $t\in[\![1:T]\!]$, 
			\item $\mathcal{U}_t$ is compact and $A_t^T+I_n$ has rank $n$ for all $t\in[\![1:T]\!]$, or
			\item $R_t$ is positive definite and $A_t^T+I_n$ has rank $n$ for all $t\in[\![1:T]\!]$.
		\end{enumerate}
		\item $\mathcal{X}_t$ is compact or $Q_t$ is positive definite for every $t\in[\![1:T]\!]$ and either
		\begin{enumerate}
			\item The intersection between the solution set of \eqref{eq_sys} and the set $\ri(\mathcal{X}_1)\times\ri(\mathcal{U}_1)\times\ldots\times\ri(\mathcal{U}_{T-1})\times\ri(\mathcal{X}_T)$ is nonempty. In particular,  $0\in \ri(\mathcal{X}_t),0\in \ri(\mathcal{U}_t)$ for all $t\in[\![1:T]\!]$, or
			\item $0\in\mathop{\rm int}(\mathcal{X}_t),0\in\mathcal{U}_t$ and  $A_t$ has rank $n$ for every $t\in[\![1:T]\!]$.
		\end{enumerate}
		\item $\mathcal{U}_t$ is compact or $R_t$ is positive definite and $A_t^T+I_n$ has rank $n$ for all $t\in[\![1:T]\!]$, and either
		\begin{enumerate}
			\item The intersection between the solution set of \eqref{eq_sys} and the set $\ri(\mathcal{X}_1)\times\ri(\mathcal{U}_1)\times\ldots\times\ri(\mathcal{U}_{T-1})\times\ri(\mathcal{X}_T)$ is nonempty. In particular,  $0\in \ri(\mathcal{X}_t),0\in \ri(\mathcal{U}_t)$ for all $t\in[\![1:T]\!]$, or
			\item $0\in\mathop{\rm int}(\mathcal{X}_t),0\in\mathcal{U}_t$ and $A_t$ has rank $n$ for all $t\in[\![1:T]\!]$.
		\end{enumerate}
	\end{enumerate}
\end{prop}
Note that there are many possible combinations of the given conditions over time that also allow the converse to be true, but we omit the details.

\subsection{Linear-convex problem with mixed constraints}
The second model we study is a linear-convex problem with mixed constraints, that is, constraints that affect jointly the state and control of the system (at the same time). Following an analogous procedure, we first show that this model can be stated as in our framework.  Throughout this part, $\ell_0,\ldots,\ell_{T-1}:\bR^{n}\times\bR^{m}\to\bR$ and $f_0,\ldots,f_{T-1}:\bR^{n}\times\bR^{m}\to\bR$ are given real-valued convex functions,

The optimal control problem we deal with now is the following:
\begin{equation}\tag{P$_\text{mix}$}\label{prob:OC}
	\left\{	\begin{array}{lll}
		\text{ Minimize }&\displaystyle\sum_{t=\tau}^{T-1}\ell_t(x_t,u_t)+g(x_T),\\
		\text{ over all }& \{x_t\}_{t=\tau}^T\subset\bR^n\text{ and }\{u_t\}_{t=\tau}^{T-1}\subset\bR^n\text{ with }x_\tau=\xi,\\
		\text{ such that }& \Delta x_{t+1}= A_tx_{t}+B_tu_{t}+\varphi_t,\quad\forall t\in[\![\tau:T-1]\!],\\
		&f_t(x_t,u_t)\leq 0,\qquad\forall t\in[\![\tau:T-1]\!],\\
		&u_t\in \mathcal{U}_t,\qquad\forall t\in[\![\tau:T-1]\!],\\
		&x_t\in \mathcal{X}_t,\qquad\forall t\in[\![\tau:T]\!].\\
	\end{array}\right.
\end{equation}
We define the Lagrangian associated to \eqref{prob:OC} as 
\[L_{t+1}(x,v)=\inf_{u\in \R^m}\left\{\ell_t(x,u)+\delta_{\mathcal{U}_t}(u)+\delta_{D_t}(x,u) +\delta_{S_t}(x,u,v)\right\} + \delta_{\mathcal{X}_t}(x),\] where $D_t=\{(x,u) : f_t(x,u)\leq 0\}$. It is not difficult to see that $\val(\rm{P_{mix}})=\vartheta_{\tau}(\xi)$.

However, we still need $L_t$ to satisfy the conditions of a convex Bolza formulation. 
\begin{enumerate}[label=\alph*)]
	\item {\it Convexity}. Write $L_{t+1}$ as $ G_t(x,v) + \delta_{\mathcal{X}_t}(x)$, where $G_t(x,v)=\inf_{u\in \R^m}g_t(x,v,u)$ and $g_t(x,v,u)=\ell_t(x,u)+\delta_{\mathcal{U}_t}(u)+\delta_{D_t}(x,u) +\delta_{S_t}(x,u,v)$. Then, as $g_t$ is clearly convex, $G_t$ is convex too, and so $L_{t+1}$ is convex.
\end{enumerate}
From now on, we make the following assumption 
\begin{equation}\label{A}
	\forall t,\, \exists \psi_t:\R^n\to \R\ \textrm{u.s.c.} \quad \textrm{s.t.} \quad \sup_{u\in\mathcal{U}_t}\{|u|\mid f_t(x,u)\leq 0\}\leq \psi_t(x) \quad \forall x\in\mathcal{X}_t, \tag{A}
\end{equation}
where u.s.c. stands for upper semicontinuous functions.  This is the discrete version of \cite[Assumption (A5)]{BecHer22}. Define $\Omega_t=\{(x,u):x\in\mathcal{X}_t,u\in\mathcal{U}_t,\ f_t(x,u)\leq 0\}$. 
\begin{enumerate}[label=\alph*),start=2]
	\item {\it Properness}. It is clear that $L_{t+1}\not\equiv +\infty$ as long as there is a feasible solution to the problem. We show that we also have that there is no $(x,v)$ such that $L_{t+1}(x,v)=-\infty$. For the sake of contradiction, suppose that there is one. Then there is a sequence $\{u_k\}$ such that $(x,u_k)\in \Omega_t$ for every $k\in \bN$ and $\ell_t(x,u_k)\to -\infty$. By \eqref{A}, $\{u_k\}$ is bounded by $\psi_t(x)$, which contradicts the continuity of $\ell_t$ ($\ell_t$ is a real valued convex function). Therefore, $L_{t+1}$ is proper.
	
	\item {\it Lower semi-continuity}. To prove that $L_{t+1}$ is l.s.c., we rely on \cite[Chapter 5]{RoyWetBook22}.  We first show that $g_t$ is level-bounded in $u$ locally uniformly in $(x,v)$, as defined in \cite{RoyWetBook22}. Let $(\bar{x},\bar{v})\in \bRn\times \bRn$ and $\alpha\in \R$. Fix $\varepsilon>0$ and let $M>0$ such that $\sup_{x\in \mathbb{B}(\bar{x},\varepsilon)}\psi_t(x) \leq M$ (its existence is guaranteed by the upper semi-continuity of $\psi_t$). We define $B=\mathbb{B}_{\R^m}(0;M)$, the ball in $\R^m$, centered at 0 with radius $M$.

	If $u\in \{u\mid g(x,v,u)\leq \alpha\}$, with $x\in \mathbb{B}(\Bar{x},\varepsilon)$ and $v\in \R^n$, then $u\in\mathcal{U}_t$ and $f_t(x,u)\leq 0$. By \eqref{A}, 
	\begin{equation*}
		|u| \leq \psi_t(x) \leq M,
	\end{equation*}
	and we have that $u\in B$. Therefore, $\{u\mid g(x,v,u)\leq \alpha\}\subset B$ for every $(x,v)\in \mathbb{B}((\bar{x},\bar{v}),\varepsilon)$ and then $g_t$ is level-bounded in $u$ locally uniformly in $(x,v)$.

	We prove now that $G_t$ is l.s.c. Let  $(\bar{x},\bar{v})\in \dom(G_t)$ and $(x_k,v_k)\to (x,v)$. We consider two cases. First, suppose that there is a subsequence such that $G_t(x_{k_j},v_{k_j})<+\infty$ for every $j\in \bN$. By continuity of $\ell_t$, we can choose $\{(x_{k_j},v_{k_j})\}$ so that $G_t(x_{k_j},v_{k_j})<M$ for some $M>0$ large enough. Then, $\{G_t(x_{k_j},v_{k_j})\}$ is bounded from above. As we have already proved that $g_t$ is level-bounded in $u$ locally uniformly in $(x,v)$, by \cite[Proposition 5.4]{RoyWetBook22} we obtain that the sequence $g_t(x_{k_j},v_{k_j},\cdot)$ is tight in the sense of \cite[Definition 5.3]{RoyWetBook22}. With this, we can follow the steps on \cite[Theorem 5.6(a)]{RoyWetBook22} to prove that 
	\begin{equation*}
		\liminf_{k\to\infty} (G_t(x_k,v_k)) = \liminf_{j\to\infty} (G_t(x_{k_j},v_{k_j}) \geq G_t(\Bar{x},\Bar{v}).
	\end{equation*}
	In the second case, there exists a natural number $n$ from which every element of the sequence is $+\infty$. It follows that
	\begin{equation*}
		\liminf_{k\to\infty} (G_t(x_k,v_k)) = +\infty  \geq G_t(\Bar{x},\Bar{v})
	\end{equation*}
	Therefore, $G_t$ is l.s.c. in $(\Bar{x},\Bar{v})$. To conclude, we show that $\dom G_t$ is closed.
	
	Let $\{(x_k,v_k)\}\subset \dom G_t$ such that $(x_k,v_k)\to (x,v)$. This implies that for every $k\in \bN$ there is $u_k\in\mathcal{U}_t$ such that $f_t(x_k,u_k)\leq 0$ and $v_k=A_tx_k+B_tu_k+\varphi_t$. The sequence $\{x_k\}$ is bounded because it is convergent. Thus, by \eqref{A} $\{u_k\}$ is bounded too. Then, there is a convergent subsequence of $\{u_k\}$. Let $u$ be the limit of this subsequence. By continuity of $f_t$, we obtain that $f_t(x,u)\leq 0$ and $v=A_tx+B_tu+\varphi_t$, concluding that $(x,v)\in \dom G_t$. This implies that $G_t$ is lsc and so is $L_{t+1}$.
\end{enumerate}
With the above arguments, $L_{t+1}$ satisfies all of the conditions for a Bolza type formulation. Let us  now focus on studying the dual problem.  In particular we are interesting in computing $K_t$
\begin{equation*}
	\begin{split}
		K(t,p,w) & = \sup_{x,v\in \R^n} \{ x\cdot w  +  v\cdot p  -L_t(x,v)\},\\
		& = \sup_{\substack{(x,u)\in \Omega_t \\ v=A_tx+B_tu+\varphi_t}}\{ x\cdot w  +  v\cdot p  -\ell_t(x,u)\},\\
		& = \sup_{(x,u)\in \Omega_t} \{x\cdot (w+A_t^Tp) +  u\cdot B_t^Tp  -\ell_t(x,u)\} +  \varphi_t\cdot p .
	\end{split}
\end{equation*}

Our focus now is on establishing conditions on the primal problem to satisfy \eqref{hyp:qual_prim} and \eqref{hyp:qual_dual}. Firstly, we proceed to calculate $\Gamma_L(t+1,x)$:

\begin{equation*}
	\begin{split}
		\Gamma_L(t+1,x) & = \{v\in \bRn \mid L_{t+1}(x,v)\in \bR\}, \\
		& = \left\{v\in \bRn \mid  \inf_{u\in\mathcal{U}_t}\{\ell_t(x,u): f_t(x,u)\leq 0 \, \wedge \, v=A_tx+B_tu+\varphi_t\}\in \bR\right\},
	\end{split}
\end{equation*}
for $x\in\mathcal{X}_t$. If $x\notin\mathcal{X}_t$, then $\Gamma_L(t+1,x)=\emptyset$ trivially. By \eqref{A}, we have that $\forall x\in\mathcal{X}_t$:
\begin{equation*}
	\inf_{u\in\mathcal{U}_t}\{\ell_t(x,u):f_t(x,u)\leq 0 \} \geq \inf_{|u|\leq \psi_t(x)}\{\ell_t(x,u)\} > -\infty.
\end{equation*}
Therefore, if $x\in\mathcal{X}_t$:
\begin{equation*}
	\begin{split}
		v\in \Gamma_L(t+1,x) & \Leftrightarrow \exists u\in\mathcal{U}_t \ \textrm{s.t.} \ v=A_tx+B_tu+\varphi_t \, \wedge \, f_t(x,u)\leq 0, \\
		& \Leftrightarrow v\in B_t(\mathcal{U}_t\cap \{u: f_t(x,u)\leq 0\})+A_tx+\varphi_t.
	\end{split}
\end{equation*}
As the set $\{u: f_t(x,u)\leq 0\}$ is clearly convex, we note that $\ri(\Gamma_L(t+1,x))=B_t(\ri(\mathcal{U}_t)\cap \ri(\{u: f_t(x,u)\leq 0\}))+A_tx+\varphi_t$ and thus it is always nonempty. With this, 
\begin{equation*}
	\X(t) =\mathcal{X}_t \cap \{x\mid \exists u\in\mathcal{U}_t \, \textrm{s.t.} \ f_t(x,u)\leq 0\}. 
\end{equation*}
By convexity of $\mathcal{U}_t$ and $f_t$, the set $\{x\mid \exists u\in\mathcal{U}_t \, \textrm{s.t.} \ f_t(x,u)\leq 0\}$ is convex. It follows that $\ri(\X(t)) = \ri(\mathcal{X}_t)\cap \ri(\{x\mid \exists u\in\mathcal{U}_t \, \textrm{s.t.} \ f_t(x,u)\leq 0\})$.

Considering this, we can address the system \eqref{eq_sys} explored in the problem \eqref{eq:LQ} to seek a sequence that meets \eqref{hyp:qual_prim}, provided there exists at least one solution satisfying a Slater condition $f_t(x_t,u_t)<0$ for all $t$. The arguments for the equilibrium point and the specific cases examined in the previous section are applicable to this problem as well.

Next, our attention shifts to establishing conditions for \eqref{hyp:qual_dual}. Note that:
\[\Gamma_K(t+1,p)  =  \bigg \{w\in \bRn \mid  \sup_{(x,u)\in \Omega_t}\{ x\cdot(w+A_t^Tw+A_t^Tp) \\
+  u\cdot (B_t^Tp+B_t^Tw) -\ell_t(x,u)\}\in \bR \bigg \},
\]
is convex. In fact, let $w_1,w_2\in \Gamma_K(t+1,p)$ and $\lambda\in [0,1]$. We have that
\begin{equation*}
	\begin{split}
		&\sup_{(x,u)\in \Omega_t}\{ x\cdot\left(\lambda w_1+(1-\lambda)w_2+A_t^T(\lambda w_1+(1-\lambda)w_2)+A_t^Tp\right) \\
		&\qquad + u\cdot \left(B_t^Tp+B_t^T(\lambda w_1+(1-\lambda)w_2)\right) -\ell_t(x,u)\} \\ 
		&= \lambda \sup_{(x,u)\in \Omega_t}\{x\cdot(w_1+A_t^Tw_1+A_t^Tp) +  u\cdot (B_t^Tp+B_t^Tw_1) -\ell_t(x,u)\}\\
		&\qquad +(1-\lambda)\sup_{(x,u)\in \Omega_t}\{ x\cdot (w_2+A_t^Tw_2+A_t^Tp) +  u\cdot (B_t^Tp+B_t^Tw_2) -\ell_t(x,u)\} \in \bR,
	\end{split}
\end{equation*}
and so $\lambda w_1+(1-\lambda)w_2\in \Gamma_K(t+1,p)$. Hence, as done before, we seek conditions that ensure $\mathbf{P}(t)=\bRn$ for any $t$, because these conditions will imply \eqref{hyp:qual_dual}.

Firstly, it's worth noting that if $\mathcal{X}_t$ is compact, we can select $w=-p$ for every $p\in \bRn$, resulting in
\begin{equation*}
	\begin{split}
		\sup_{(x,u)\in \Omega_t}&\{ x\cdot(w+A_t^Tw+A_t^Tp) +  u\cdot (B_t^Tp+B_t^Tw) -\ell_t(x,u)\} = \\
		&\sup_{(x,u)\in \Omega_t}\{\langle x,p+2A_t^Tp\rangle  -\ell_t(x,u)\} \in \bR,
	\end{split}
\end{equation*}
by continuity of $\ell_t$. Therefore, \eqref{hyp:qual_dual} holds true in this case. Similarly, \eqref{hyp:qual_dual} also holds if $\mathcal{U}_t$ is compact and $A_t^T+I_n$ has rank $n$ for every $t$.

By continuity of $\ell_t$, we also have that \eqref{hyp:qual_dual} holds if $\Omega_t$ is compact. This is true if, for example, $f_t$ is coercive in $x$. To prove this, for the sake of contradiction let $\{(x_k,u_k)\}\subset \Omega_t$ such that $|(x_k,u_k)|\to \infty$. We have two cases. Either $|x_k|\to \infty$, which contradicts the coercivity of $f_t$ in $x$, or $|u_k|\to \infty$ and $\{x_k\}$ is bounded, which contradicts \eqref{A}.

On the other hand, it is easy to see that if $\ell_t$ is coercive for any $t$, then \eqref{hyp:qual_dual} holds. If the matrix $A_t^T+I_n$ has rank $n$ for all $t$, we can refine this argument by asking
\begin{gather*}\label{B} \tag{B}
	\exists \kappa_0\in \bR,h_t:\bR^m\to \bR \ \textrm{such that} \\
	\quad \sup_{(x,u)\in \Omega_t}\{|x|:\ell_t(x,u)-\langle z,u\rangle\leq \kappa_0\} \leq h_t(z), \quad \forall z\in \bR^m.
\end{gather*}
In this case, $\mathbf{P}(t)=\bRn$ for every $t$ as well. To prove this, let $p\in \bRn$ and $w\in \bRn$ such that $w+A_t^Tw=-A_t^Tp$. For the sake of contradiction, suppose that $\exists \{(x_k,u_k)\}\subset \Omega_t$ such that 
\begin{equation*}
	x_k\cdot (w+A_t^Tw+A_t^Tp) + u_k\cdot (B_t^Tp+B_t^Tw) -\ell_t(x_k,u_k) =  u_k\cdot (B_t^Tp+B_t^Tw) -\ell_t(x_k,u_k) \to \infty.
\end{equation*}
This implies that there exists some natural $N_0$ such that 
$$u_k\cdot (B_t^Tp+B_t^Tw) -\ell_t(x_k,u_k)\geq -\kappa_0,\qquad \forall k\geq N_0.$$ By taking $z=B_t^Tp+B_t^Tw$ in \eqref{B}, then $\sup_{k\geq N_0}|x_k|\leq h_t(z)$. Then, by \eqref{A}, there exists some $M\in \bR$ such that $\sup_{k\geq N_0}|u_k|\leq M$, which is a contradiction with $u_k\cdot (B_t^Tp+B_t^Tw) -\ell_t(x_k,u_k) \to \infty$ as $\ell_t$ is continuous.

The condition \eqref{B} corresponds to the discrete version of \cite[Assumption (A6)]{BecHer22}, and it is satisfied, for example, if $\ell_t$ is coercive in $u$ and has at least linear growth in $x$.

The main results of this subsection are summarized in the following proposition:
\begin{prop}
	Under condition \eqref{A}, the problem \eqref{prob:OC} admits a Bolza problem formulation and, therefore, Theorem \ref{thm:characteristics} holds true. Moreover, the converse remains true if the intersection between the solution set of \eqref{eq_sys} and the set $\ri(\mathcal{X}_1)\times\ri(\mathcal{U}_1)\times\ldots\times\ri(\mathcal{U}_{T-1})\times\ri(\mathcal{X}_T)$ is nonempty, and at least one of its elements satisfies the condition $f_t(x_t,u_t)<0$ for every $t\in[\![1:T]\!]$ and either of the following holds:
\begin{enumerate}[label=\alph*)]
		\item $\mathcal{U}_t$ is compact and the matrix $A_t^T+I_n$ has rank $n$, $\mathcal{X}_t$ is compact or $f_t$ is coercive in $x$ for all $t\in[\![1:T]\!]$.
		\item Alternatively, $\ell_t$ is coercive or the matrix  $A_t^T+I_n$ has rank $n$ and $\ell_t$ satisfies \eqref{B} for all $t\in[\![1:T]\!]$.
	\end{enumerate}
\end{prop}

\section{Conclusion}
In this paper we have shown that the evolution over time of the subgradients of the value function of a (deterministic) discrete-time convex Bolza problem can be associated with trajectories of a discrete-time Hamiltonian system. This is the so-called \emph{discrete-time characteristic method}. We have also demonstrated that the qualificiation conditions required for our results to hold are satisfied in several cases of interest; e.g. the Linear-Quadratic problem and the Linear-Convex problem with mixed constraints.

This paper paves the way for future extension to stochastic convex Bolza problems. Indeed, our main result (Theorem \ref{thm:characteristics}) as well as our intermedediate results (Proposition \ref{prop:conjugacy_1} and \ref{prop:conjugacy_2}) are consequences of the research reported by Rockafellar and Wets in \cite{RocWet83}. For our paper, we have only used the deterministic part, but we plan to use the stochastic part in a forthcoming paper.

\bibliographystyle{abbrv}
\bibliography{biblio}

\begin{thebibliography}{10}

\bibitem{AusTebBook03}
A.~Auslender and M.~Teboulle.
\newblock {\em Asymptotic {C}ones and {F}unctions in {O}ptimization and
  {V}ariational {I}nequalities}.
\newblock Springer-Verlag, 2003.

\bibitem{BecHer22}
J.~Becerril and C.~Hermosilla.
\newblock Optimality conditions for linear-convex optimal control problems with
  mixed constraints.
\newblock {\em Journal of Optimization Theory and Applications},
  194(3):795--820, 2022.

\bibitem{Goe04b}
R.~Goebel.
\newblock Regularity of the optimal feedback and the value function in convex
  problems of optimal control.
\newblock {\em Set-Valued Analysis}, 12(1-2):127--145, 2004.

\bibitem{Goe04a}
R.~Goebel.
\newblock Stationary {H}amilton--{J}acobi equations for convex control
  problems: uniqueness and duality of solutions.
\newblock In {\em Optimal control, stabilization and nonsmooth analysis}, pages
  313--322. Springer, 2004.

\bibitem{Goe05}
R.~Goebel.
\newblock Convex optimal control problems with smooth {H}amiltonians.
\newblock {\em SIAM Journal on Control and Optimization}, 43(5):1787--1811,
  2005.

\bibitem{Goe05a}
R.~Goebel.
\newblock Duality and uniqueness of convex solutions to stationary
  {H}amilton--{J}acobi equations.
\newblock {\em Transactions of the American Mathematical Society},
  357(6):2187--2203, 2005.

\bibitem{RocGoe08}
R.~Goebel and R.~T. Rockafellar.
\newblock Linear-convex control and duality.
\newblock In {\em Geometric Control And Nonsmooth Analysis: In Honor of the
  73rd Birthday of H Hermes and of the 71st Birthday of RT Rockafellar}, pages
  280--299. World Scientific, 2008.

\bibitem{GoeSub07}
R.~Goebel and M.~Subbotin.
\newblock Continuous time linear quadratic regulator with control constraints
  via convex duality.
\newblock {\em IEEE Transactions on Automatic Control}, 52(5):886--892, 2007.

\bibitem{HerWol19}
C.~Hermosilla and P.~Wolenski.
\newblock A characteristic method for fully convex bolza problems over arcs of
  bounded variation.
\newblock {\em SIAM Journal on Control and Optimization}, 57(4):2873--2901,
  2019.

\bibitem{HerWol17}
C.~Hermosilla and P.~R. Wolenski.
\newblock Constrained and impulsive linear quadratic control problems.
\newblock In {\em Proceeding of thye 20th World Congress of the IFAC},
  volume~50, pages 1637--1642, 2017.

\bibitem{PenPer14}
T.~Pennanen and A.-P. Perkki\"o.
\newblock Duality in convex problems of {B}olza over functions of bounded
  variation.
\newblock {\em SIAM Journal on Control and Optimization}, 52(3):1481--1498,
  2014.

\bibitem{Roc70}
R.~T. Rockafellar.
\newblock Conjugate convex functions in optimal control and the calculus of
  variations.
\newblock {\em Journal of Mathematical Analysis and Applications},
  32(1):174--222, 1970.

\bibitem{RocBook70}
R.~T. Rockafellar.
\newblock {\em Convex {A}nalysis}, volume~28.
\newblock Princeton University Press, 1970.

\bibitem{Roc71}
R.~T. Rockafellar.
\newblock Existence and duality theorems for convex problems of {B}olza.
\newblock {\em Transactions of the American Mathematical Society}, 159(1--40),
  1971.

\bibitem{Roc72}
R.~T. Rockafellar.
\newblock State constraints in convex control problems of {B}olza.
\newblock {\em SIAM Journal on Control and Optimization}, 10(4):691--715, 1972.

\bibitem{Roc73}
R.~T. Rockafellar.
\newblock Saddle points of {H}amiltonian systems in convex problems of
  {L}agrange.
\newblock {\em Journal of Optimization Theory and Applications}, 12:367--390,
  1973.

\bibitem{Roc76}
R.~T. Rockafellar.
\newblock Dual problems of {L}agrange for arcs of bounded variation.
\newblock In {\em Calculus of variations and control theory (Proc. Sympos.,
  Math. Res. Center, Univ. Wisconsin, Madison, Wis., 1975; dedicated to
  Laurence Chisholm Young on the occasion of his 70th birthday)}, pages
  155--192. Academic Press, New York, 1976.

\bibitem{Roc87}
R.~T. Rockafellar.
\newblock Linear-quadratic programming and optimal control.
\newblock {\em SIAM Journal on Control and Optimization}, 25(3):781--814, 1987.

\bibitem{Roc89}
R.~T. Rockafellar.
\newblock Hamiltonian trajectories and duality in the optimal control of linear
  systems with convex costs.
\newblock {\em SIAM Journal on Control and Optimization}, 27(5):1007--1025,
  1989.

\bibitem{Roc04}
R.~T. Rockafellar.
\newblock Hamilton-{J}acobi theory and parametric analysis in fully convex
  problems of optimal control.
\newblock {\em Journal of Global Optimization}, 28(3-4):419--431, 2004.

\bibitem{RocWet83}
R.~T. Rockafellar and R.~Wets.
\newblock Deterministic and stochastic optimization problems of {B}olza type in
  discrete time.
\newblock {\em Stochastics: An International Journal of Probability and
  Stochastic Processes}, 10(3-4):273--312, 1983.

\bibitem{RocWol00}
R.~T. Rockafellar and P.~R. Wolenski.
\newblock Convexity in {H}amilton-{J}acobi theory {I}: {D}ynamics and duality.
\newblock {\em SIAM Journal on Control and Optimization}, 32(2):442--470, 2000.

\bibitem{RocWol00b}
R.~T. Rockafellar and P.~R. Wolenski.
\newblock Convexity in {H}amilton-{J}acobi theory {II}: Envelope
  representations.
\newblock {\em SIAM Journal on Control and Optimization}, 39(5):1351--1372,
  2000.

\bibitem{RoyWetBook22}
J.~O. Royset and R.~J.-B. Wets.
\newblock {\em An optimization primer}.
\newblock Springer, 2022.

\end{thebibliography}

\end{document}